\documentclass[11pt,twoside]{article}
\usepackage[utf8]{inputenc}
\usepackage[english]{babel}
\usepackage{fancyhdr}
\usepackage{amssymb,amsmath,amsthm}
\usepackage{mathtools}
\usepackage{color}
\usepackage{hyperref}
\usepackage{caption}
\usepackage{algorithm}
\usepackage{algorithmicx}
\usepackage{algpseudocode}
\usepackage[inline]{enumitem}
\usepackage[square,numbers]{natbib}

\DeclareMathOperator*{\argmin}{argmin}

\DeclareMathOperator*{\vol}{Vol}
\DeclarePairedDelimiter{\bigO}{\mathcal O(}{)}

\def\fmin{\ensuremath{f^{\text{min}}}}
\def\Keps{\ensuremath{K_{\epsilon}}}

\def\Lmax{\ensuremath{L^{\text{max}}}}
\def\R{\ensuremath{\mathbb R}}
\def\Rn{\ensuremath{\mathbb R^n}}
\def\Rm{\ensuremath{\mathbb R^m}}
\def\eps{\epsilon}

\newtheorem{assumption}{Assumption}
\newtheorem{proposition}{Proposition}
\newtheorem{definition}{Definition}

\newtheorem{lemma}{Lemma}
\newtheorem{theorem}{Theorem}
\newtheorem{remark}{Remark}

\hypersetup{
    colorlinks=false,
    urlcolor=black,
    pdfborder={0 0 0}
}

\newcommand{\email}[1]{E-mail: \href{mailto:#1}{\texttt{#1}}}

\begin{document}
\thispagestyle{plain}

\setcounter{page}{1}

{\centering
{\LARGE \bfseries Worst-Case Complexity of High-Order Algorithms for Pareto-Front Reconstruction}

\bigskip\bigskip
Andrea Cristofari$^*$, Marianna De Santis$^\dag$, Stefano Lucidi$^\ddag$, Giampaolo Liuzzi$^\ddag$
\bigskip

}

\begin{center}
\small{\noindent$^*$Department of Civil Engineering and Computer Science Engineering \\
University of Rome ``Tor Vergata'' \\
Via del Politecnico, 1, 00133 Rome, Italy \\
\email{andrea.cristofari@uniroma2.it} \\
\bigskip
$^\dag$Department of Information Engineering \\
Universit\`a degli studi di Firenze \\
Via di Santa Marta 3, 50139 Firenze, Italy \\
\email{marianna.desantis@unifi.it} \\
\bigskip
$^\ddag$Department of Computer, Control and Management Engineering \\
Sapienza Universit\`a di Roma \\
Via Ariosto 25, 00185 Roma, Italy \\
\email{liuzzi@diag.uniroma1.it}, \email{lucidi@diag.uniroma1.it} \\
}
\end{center}

\bigskip\par\bigskip\par
\noindent \textbf{Abstract.}
In this paper, we are concerned with a worst-case complexity analysis of a-posteriori algorithms for unconstrained multiobjective optimization.
Specifically, we propose an algorithmic framework that generates sets of points by means of $p$th-order models regularized with a power $p+1$ of the norm of the step.
Through a tailored search procedure, several trial points are generated at each iteration and they can be added to the current set if a decrease is obtained for at least one objective function.
Building upon this idea, we devise two algorithmic versions: at every iteration, the first tries to update all points in the current set, while the second tries to update only one point.
Under Lipschitz continuity of the derivatives of the objective functions, we derive worst-case complexity bounds for both versions.
For the first one, we show that at most $\bigO{\epsilon^{-m(p+1)/p}}$ iterations are needed to generate a set where all points are $\epsilon$-approximate Pareto-stationary, with $m$ being the number of objective functions, requiring at most $\bigO{|X(\eps)|\epsilon^{-m(p+1)/p}}$ function evaluations, where $X(\eps)$ is the largest set of points computed by the algorithm. Additionally, at most $\bigO{\epsilon^{-(p+1)/p}}$ iterations are needed to generate at least one $\epsilon$-approximate Pareto-stationary point, requiring at most $\bigO{|X(\eps)|\epsilon^{-(p+1)/p}}$ function evaluations.
For the second version, we get $\bigO{\epsilon^{-m(p+1)/p}}$ worst-case complexity bounds on the number of iterations and function evaluations for generating at least one $\epsilon$-approximate Pareto-stationary point.
Our results align with those for single objective optimization and generalize those obtained for methods that produce a single Pareto-stationary point.


\bigskip\par
\noindent \textbf{MSC2000 subject classifications.} 90C29. 90C30. 65Y20.

\section{Introduction}
In this paper, we consider unconstrained multiobjective  optimization problems of the following form:
\begin{equation}\label{prob}
\min_{x \in \R^n} \,F(x) := (f_1(x), \ldots, f_m(x))^T,
\end{equation}
with (possibly non-convex) objective functions $f_i \colon \Rn \to \R$, $i=1,\ldots,m$ and $m \ge 2$.

For the solution of Problem~\eqref{prob}, where we want to simultaneously minimize a number of conflicting objective functions, a common approach
is represented by the so-called descent methods. They are characterized by the use of suitable directions that allow all objective functions to decrease when using an appropriate stepsize.
Such a paradigm generates a sequence of points converging, over appropriate subsequences, to a Pareto-stationary point.

In particular, for first-order descent methods, a steepest descent method was originally proposed in~\cite{fliege:2000}, computing a descent direction by minimizing a first-order model with second-order regularization. When the gradients of the objective functions are Lipschitz continuous, the stationarity violation was proved in~\cite{fliege:2019,lapucci:2024} to converge to zero at a $\bigO{k^{-1/2}}$ rate and provide a worst-case iteration complexity of $\bigO{\epsilon^{-2}}$ to produce an approximate solution with precision $\epsilon > 0$.
Extending that approach to the use of higher order derivatives, a descent algorithm was proposed in~\cite{calderon:2022} where, at each iteration, a descent direction can be computed by approximately minimizing a $p$th-order model regularized with a power $p+1$ of the norm of the step.
Assuming that the $p$th derivatives of the objective functions are H{\"o}lder continuous with parameter $\beta \in [0,1]$, it was shown a worst-case iteration complexity of $\bigO{\epsilon^{-(p+\beta)/(p+\beta-1)}}$ for generating an approximate solution with precision $\epsilon>0$, thus yielding to a $\bigO{\epsilon^{-(p+1)/p}}$ worst-case complexity when the $p$th derivatives of the objective function are Lipschitz continuous (i.e., $\beta=1$).
These complexity bounds align with those obtained for single-objective optimization (see, e.g.,~\cite{birgin:2017,cartis:2011,cartis:2019,grapiglia:2015,nesterov:2006}).

The above mentioned schemes for multiobjective optimization are able to produce a single Pareto-stationary point. This may be unsatisfactory when the goal is to approximate the entire Pareto set rather than computing just a single point.
For this reason, several algorithms that build sequences of sets were recently proposed in the literature~\cite{cocchi2020augmented,cocchi2021pareto,cocchi:2020,custodio:2021,custodio2011direct,lapucci:2023,lapucci2024effective,liuzzi2025worst,mohammadi2024trust}.

Following this trend, here we design and analyze an algorithmic approach that builds sequences of sets of points
by means of high-order regularized models.
In the proposed scheme, we are given a set of points $X_k$ at every iteration $k$ and, chosen any point in $X_k$ (possibly all), we try to produce new points to be added to the set by minimizing an appropriate model where the $p$th-order Taylor approximation is regularized by a power $p+1$ of the norm of the step.
Specifically, to select appropriate values for the regularization parameters, we generate several trial points using a search procedure that repeatedly increases the regularization parameters until a dominating point is obtained,  that is, such that all the objective values are lower than the initial ones.
A distinguishing feature of our approach is that not all the trial points computed in that search phase are discarded. Instead, they can be added to the current set of points if they make \textit{at least one objective function} decrease.

For this paradigm, we consider two algorithmic versions: at every iteration $k$, the first version tries to update all points at its disposal in the current set $X_k$, while the second version chooses only one point from $X_k$.
Loosely speaking, the two versions meet two different needs: the first one aims to produce a more accurate approximation of the Pareto front but requires a major computational effort in practice. Conversely, the second strategy is more efficient from a computational standpoint.

Under Lipschitz continuity of the derivatives of the objective functions,
for both versions we derive worst-case iteration and function evaluation complexity bounds for generating a set of approximate Pareto-stationary points. 
To the best of the authors' knowledge, this represents the first attempt to provide a worst-case analysis for algorithms that generate sets of points and use high order derivatives.

\subsection{Main Contributions}
Let us summarize the main contributions of the paper.

\begin{enumerate}
    \item We devise an algorithmic framework for Pareto front reconstruction in unconstrained multiobjective optimization using high-order regularized models. Specifically, we give two algorithmic versions which differ in how the current set of points is updated at each iteration: while the first version tries to update each point of the set, the second one---a lighter version---tries to update only one arbitrarily chosen point of the set. As discussed earlier, these two versions are expected to achieve different levels of accuracy in the Pareto front reconstruction and different efficiency.
    \item For both algorithmic versions, we provide a worst-case iteration and function evaluation complexity analysis for generating a set containing approximate Pareto-stationary points. Our results show that, for the algorithm that updates all points at its disposal at every iteration, we generate a set where all points are $\epsilon$-approximate Pareto-stationary after at most $\bigO{\epsilon^{-m(p+1)/p}}$ iterations, requiring at most $\bigO{|X(\eps)|\epsilon^{-m(p+1)/p}}$ function evaluations, where $X(\eps)$ is the largest set of points computed by the algorithm. Moreover, to generate at least one $\epsilon$-approximate Pareto-stationary point, we show a worst-case complexity of $\bigO{\epsilon^{-(p+1)/p}}$ for the number of iterations and $\bigO{|X(\eps)|\epsilon^{-(p+1)/p}}$ for the number of function evaluations.  For the light algorithm, which is much less expensive, only the latter type of complexity bounds can be established (i.e., to generate at least one $\epsilon$-approximate Pareto-stationary point), which are $\bigO{\epsilon^{-m(p+1)/p}}$ for both the number of iterations and the number of function evaluations.
\end{enumerate}
Moreover, we investigate the use of approximate conditions to solve the (possibly non-convex) high-order regularized models inexactly, thus mitigating the difficulty connected with their practical solution. Our analysis shows that the complexity bounds remain the same 
(up to the constant factors) when approximate minimizers are used.

\subsection{Notation and contents}
Given a vector $v \in \R^n$, we denote by $v_i$ the $i$th entry of $v$, while $\|v\|$ denotes the Euclidean norm of $v$.
Given two vectors $u,v\in\R^m$, then $u \le v$ ($u < v$) if and only if $u_i\le v_i$ ($u_i < v_i$) for all $i=1,\dots,m$.
Given two vectors $x,y\in\R^n$, we say that $x$ dominates $y$, and we write $x\prec y$, if $F(x)\le F(y)$ and $F(x)\neq F(y)$. Given a set of points $X\subseteq\R^n$, we define $F(X) := \{z\in\R^m:\ z = F(x)\ \text{for}\ x\in X\}$.
The volume of a set $X\subseteq\R^n$ is denoted by $\vol(X)$.
By $\mathbf{1}$ we denote the vector of all ones of appropriate dimensions, typically $\mathbf{1}\in\R^m$.
We indicate by $\Rm_+$ the set of vectors in $\Rm$ having all components strictly positive.

The rest of the paper is organized as follows.
In Section~\ref{sec:preliminaries}, we recall some concepts on multiobjective optimization problems and point out some results on quadratically regularized models.
In Section~\ref{sec:high_order_model}, we analyze high-order regularized models.
In Section~\ref{sec:HOP}, we describe the proposed algorithmic approach and carry out a worst-case complexity analysis, using both exact and inexact minimizers of the regularized model.
In Section~\ref{sec:LHOP}, we describe a light version of the algorithm and give worst-case complexity bounds. In Section~\ref{sec:examples}, we report some numerical examples to show the behavior of the proposed algorithms. Finally, some conclusions are drawn in Section~\ref{sec:conclusions}.

\section{Preliminaries}\label{sec:preliminaries}
In this section we introduce some preliminary material. First, let us recall the definitions of Pareto-optimality and -stationarity for Problem~\eqref{prob}.

\begin{definition}[Pareto optimality]
A point $x^* \in \Rn$ is said to be Pareto optimal (or efficient) for Problem~\eqref{prob} if
there does not exist $x \in \Rn$ such that $x\prec x^*$, that is $F(x) \le F(x^*)$,
with $F(x) \ne F(x^*)$.
\end{definition}

\begin{proposition}\label{prop:stat1}
Let $x^*\in \R^n$ be a Pareto optimal point for Problem~\eqref{prob}. Then, for any $d \in \R^n$, an index $j \in \{1,\ldots,m\}$ exists such that
 \begin{equation*}
{\nabla f_j}(x^*)^{T}d \ge 0.
 \end{equation*}
\end{proposition}

\begin{definition}[Pareto-stationary point]\label{def:stazionario}
A point $x^*\in \R^n$ is Pareto-stationary for Problem~\eqref{prob} if, for any $d \in \R^n$, an index $j \in \{1,\ldots,m\}$ exists such that
 \begin{equation*}
{\nabla f_j}(x^*)^{T}d \ge 0.
 \end{equation*}
\end{definition}

Now, starting from the above results, we can also give the definition and a characterization of efficient sets.

\begin{definition}[Efficient set]
A set of points ${\cal E}\subseteq\R^n$ is said to be an efficient set if every point $x\in\cal E$ is a Pareto optimal point of Problem~\eqref{prob}.
\end{definition}

\begin{proposition}\label{prop:statEffSet}
Let $\mathcal{E}\subseteq \R^n$ be an efficient set for Problem~\eqref{prob}. Then, for any $x\in \mathcal{E}$ and $d \in \R^n$,  there exists an index $i \in \{1,\ldots,m\}$ such that $ \nabla f_i (x)^T d \ge 0$.  
\end{proposition}
Stronger conditions to characterize the efficient set $\mathcal{E}$ were also given in~\cite{cristofari:2024}.
Here, according to Proposition~\ref{prop:statEffSet}, we give the following definition of a Pareto-stationary set. 
\begin{definition}[Pareto-stationary set]
Let  $S\subseteq \R^n$ be a non-empty set. 
We say that $S$ is a Pareto-stationary set for Problem~\eqref{prob}
if, for any $x\in S$ and $d \in \R^n$, there exists an index $i \in \{1,\ldots,m\}$ such that $ \nabla f_i (x)^Td \ge 0$.  
\end{definition}

To measure the quality of a set of points in the image space (i.e., $\R^m$) of a multiobjective problem, it is common to use the hypervolume indicator. Before giving its definition, we need the following assumption.

\begin{assumption}\label{ass:boundedness}
For all $i=1,\dots,m$, and given $x_0\in\R^n$, it holds that
\begin{align*}
-\infty < f_i^{\min} & := \inf\{f_i(x) \colon x\in\R^n \}, \\
+\infty > f_i^{\max} & := \sup\{ f_i(x) \colon x\in\R^n \text{ and } x_0\not\prec x\}.  
\end{align*}
Let us also denote $f_{\max} := (f_1^{\max},\dots,f_m^{\max})^T\in\R^m$ and $f_{\min} := (f_1^{\min},\dots,f_m^{\min})^T\in\R^m$.
\end{assumption}

\begin{definition}[Hypervolume indicator~{\cite[Definition 3.1]{custodio:2021}}]\label{HI_definition} The hypervolume indicator for some set $A\subset\R^m$ and a reference point $\rho \in \R^m$
that is dominated by all the points in $A$ is defined as
\[
HI(A) := \vol\left(\{b \in \R^m : b \le \rho\ \text{and}\ \exists a \in A : a \le b\}\right) = \vol\left(\bigcup_{a\in A}[a,\rho]\right),
\]
where $[a,\rho] = \{y\in\R^m:\ a_i\le y_i\le \rho_i,\ i=1,\dots,m\}$.
\end{definition}

A key result is expressed in the following lemma. It relates the increase of the hypervolume in the image space when, given a set of points, a new point is added which is nondominated with respect to all other points of the set.

\begin{lemma}[{\cite[Lemma 3.1]{custodio:2021}}]\label{HIincrease_gen}
Let $Y\subseteq\R^n$, $x \in \Rn \setminus Y$, and $\nu>0$ such that 
\[
F(x) \not > F(y) - \nu\mathbf{1} \quad \forall y\in Y.
\]
Let the reference point be $\rho = f_{\max}+\alpha\mathbf{1}$ (with $\alpha>0$ being sufficiently large). Then,
\[
HI(F(Y\cup\{x\})) - HI(F(Y)) \ge \nu^m. 
\]
\end{lemma}

Note that, under Assumption~\ref{ass:boundedness}, given the efficient set $\mathcal{E}\subseteq \R^n$ of Problem~\eqref{prob}, $\overline{HI}\in\R$ exists such that
\begin{equation}\label{HI}
HI (F(\mathcal{E}))\le \overline{HI}<+\infty.
\end{equation}

\subsection{Quadratically regularized models}\label{sec:quad_reg_model}
In this subsection, we give some results on quadratically regularized models applied to Problem~\eqref{prob}, laying the groundwork for the proposed algorithmic framework, which extends those concepts to higher order models dealing with sets of points.

To begin with, given $x \in \R^n$, consider the following problem:
\begin{equation}\label{1st_order_model}
\min_{v \in \R^n}{\max_{i=1,\ldots,m}{\nabla f_i(x)^T v + \frac12 \|v\|^2}}.
\end{equation}
We see that~\eqref{1st_order_model}
consists of a first-order term with quadratic regularization.
It is a strongly convex problem and has a unique optimal solution, which will be denoted by $v(x)$, that is,
\begin{equation}\label{vx}
v(x) \in \argmin_{v \in \R^n}{\max_{i=1,\ldots,m}{\nabla f_i(x)^T v + \frac12 \|v\|^2}}.
\end{equation}
From~\cite{fliege:2000}, we get the following result.
\begin{proposition}[{\cite{fliege:2000}}]\label{prop:stat_1st_order}
Given $x \in \R^n$,the following holds:
\begin{enumerate}[label=(\roman*)]
\item the optimal value of Problem~\eqref{1st_order_model} is non-positive,
\item $x^*$ is Pareto-stationary if and only if $\|v(x^*)\|= 0$, \label{prop:stat_1st_order_pareto}
\item the mapping $x\to v(x)$ is continuous.
\end{enumerate}
\end{proposition}
In view of the above proposition, we see that Problem~\eqref{1st_order_model} has a twofold role: it provides an amount of Pareto-stationarity violation at a given point $x$ and, if $v(x) \ne 0$, the latter can be used as a descent direction.

We also note that, if $m=1$, then $v(x) = -\nabla f_1(x)$, thus reducing to the anti-gradient direction for single-objective optimization.

Moreover, Problem~\eqref{1st_order_model} can be equivalently reformulated as
\begin{equation}\label{1st_order_model_eq}
\begin{split}
& \min_{(v,\alpha) \in \R^{n+1}}{\alpha + \frac12 \|v\|^2} \\
\text{s.t. } & \nabla f_i(x)^T v \le \alpha, \quad i = 1,\ldots,m.
\end{split}
\end{equation}
It is straightforward to verify that $v$ is an optimal solution of~\eqref{1st_order_model} if and only if there exists $\alpha \in \R$ such that $(v,\alpha)$ is an optimal solution of~\eqref{1st_order_model_eq}.
Moreover, for Problem~\eqref{1st_order_model_eq}, we can define the Lagrangian function
\[
L(v,\alpha,\lambda) = \alpha + \frac12 \|v\|^2 + \sum_{i=1}^m \lambda_i \Bigl(\nabla f_i(x)^T v - \alpha\Bigr)
\]
and give the following KKT conditions:
\begin{subequations}
\begin{align}
& v + \sum_{i=1}^m \lambda_i \nabla f_i(x) = 0, \label{kkt1_1st_order_model} \\
& \sum_{i=1}^m \lambda_i = 1, \label{kkt2_1st_order_model} \\
& \lambda_i (\nabla f_i(x)^T v - \alpha) = 0, \quad i = 1,\ldots,m, \label{kkt3_1st_order_model} \\
& \nabla f_i(x)^T v - \alpha \le 0, \quad i = 1,\ldots,m, \label{kkt4_1st_order_model} \\
& \lambda_i \ge 0, \quad i = 1,\ldots,m. \label{kkt5_1st_order_model}
\end{align}
\end{subequations}

Let us conclude this section by giving a key result which we will use for high-order models to analyze the first-order stationarity, expressed in terms of $\|v(x)\|$ in view of Proposition~\ref{prop:stat_1st_order}.

\begin{lemma}\label{lemma:dual}
Given $x \in \Rn$, let $v(x)$ be defined as in~\eqref{vx}.
For all $u \in \Rm$ such that $\sum_{i=1}^m u_i = 1$ and $u \ge 0$, we have that
\[
\|v(x)\| \le \biggl\|\sum_{i=1}^m u_i \nabla f_i(x)\biggr\|.
\]
\end{lemma}

\begin{proof}
Since $v(x)$ is the unique optimal solution of Problem~\eqref{1st_order_model}, then an optimal solution of~\eqref{1st_order_model_eq} is given by $(v(x),\alpha(v(x)))$, with $\alpha(v) := \max_{i=1,\ldots,m}{\nabla f_i(x)^T v}$.
Since~\eqref{1st_order_model_eq} is a linearly constrained problem, then the constraint qualification holds and 
$(v(x),\alpha(v(x)))$ satisfies the KKT conditions~\eqref{kkt1_1st_order_model}--\eqref{kkt5_1st_order_model} with a multiplier vector $\bar \lambda \in \Rm$ (see, e.g.,~\cite{bertsekas:1999}). In particular, from~\eqref{kkt1_1st_order_model}, \eqref{kkt2_1st_order_model} and~\eqref{kkt5_1st_order_model}, we can write
\begin{equation}\label{vx_kkt}
v(x) = -\sum_{i=1}^m \bar \lambda_i \nabla f_i(x), \quad \sum_{i=1}^m \bar \lambda_i = 1, \quad \bar \lambda \ge 0.
\end{equation}
Now, consider the dual of Problem~\eqref{1st_order_model_eq}, given by
\begin{equation}\label{stat_dual}
\begin{split}
& \max_{\lambda \in \Rm} -\frac12\biggr\|\sum_{i=1}^m \lambda_i \nabla f_i(x)\biggr\|^2 \\
\text{s.t. } & \sum_{i=1}^m \lambda_i = 1, \\
             & \lambda \ge 0.
\end{split}
\end{equation}
From strong duality,
it follows that $\bar \lambda$ is an optimal solution for the dual Problem~\eqref{stat_dual}.
Hence,
\[
-\biggr\|\sum_{i=1}^m \bar \lambda_i \nabla f_i(x)\biggr\|^2 \ge -\biggr\|\sum_{i=1}^m u_i \nabla f_i(x)\biggr\|^2 \quad \forall u \in \Rm \text{ such that } \sum_{i=1}^m u_i = 1 \text{ and } u \ge 0.
\]
Recalling~\eqref{vx_kkt}, the left-hand side of the above inequality is equal to $-\|v(x)\|^2$, yielding to the desired result.
\end{proof}

\section{High-order regularized models}\label{sec:high_order_model}
From an algorithmic perspective, a solution of Problem~\eqref{1st_order_model} can be used to obtain a descent direction for all objective functions (see Proposition~\ref{prop:stat_1st_order}).
Building upon this idea, a steepest descent method was defined in~\cite{fliege:2000} by moving, at each iteration $k$, from the current point $x_k$ along a direction $v(x_k)$ computed as in~\eqref{vx} and using a stepsize obtained by a line search procedure.
Under the assumption that $\nabla f_i$, $i=1,\ldots,m$, are Lipschitz continuous, a convergence rate of $\bigO{k^{-1/2}}$ for the first-order stationarity violation was proved for such a scheme in~\cite{fliege:2019}, leading to a worst-case iteration complexity of $\bigO{\epsilon^{-2}}$ to drive the first-order stationarity violation below a given $\epsilon>0$.
The same convergence rate and worst-case iteration complexity were then obtained in~\cite{lapucci:2024} for a broader class of first-order methods.

Since Problem~\eqref{1st_order_model} consists of a first-order term with quadratic regularization, it appears natural to define a more general framework using models of order $p \ge 1$ with regularization of order $p+1$, extending ideas from single-objective optimization~\cite{birgin:2017}.
Such an approach was investigated in~\cite{calderon:2022} under  H{\"o}lder continuity of the derivatives (thus including Lipschitz continuity as a special case). However, while the algorithm designed in~\cite{calderon:2022} is able to produce only a single Pareto-stationary point, here we want to define an algorithm able to approximate the Pareto front.

Let us assume that, for a given $p \ge 1$ and all $i=1,\ldots,m$,
the $p$th derivatives of $f_i$, denoted by $\nabla^p f_i$, are Lipschitz continuous with constant $L_i>0$. Namely,
\[
\|\nabla^p f_i(x)-\nabla^p f_i(y)\|_{[p]} \le L_i \|x-y\| \quad \forall x,y \in \Rn,
\]
where, using a standard notation (see, e.g., \cite{birgin:2017,cartis:2018}), 
$\|\cdot\|_{[p]}$ is the tensor norm recursively induced on the space of $p$th-order tensors, which is given by
\[
\|T\|_{[p]} := \max_{\|v_1\| = \dots = \|v_p\|=1} |T[v_1,\ldots,v_p]|,
\]
with $T[v_1, \cdots, v_j]$ indicating the tensor of order $p-j \ge 0$ resulting from the application of the $p$th-order tensor $T$ to the vectors $v_1,\ldots,v_j$.

Then, given a point $x \in \Rn$ and a regularization parameter vector $\sigma \in \Rm_+$, we can consider the following problem:
\begin{equation}\label{high_order_model}
\min_{s \in \R^n}{\max_{i=1,\ldots,m}{\sum_{j=1}^p \frac 1{j!} \nabla^j f_i(x)[s]^j + \frac{\sigma_i}{p!} \|s\|^{p+1}}},
\end{equation}
where $\nabla^j f_i(x)[s]^j$ stands for $\nabla^j f_i(x)$ applied $j$ times to the vector $s$.

To simplify the notation, given $x \in \Rn$ and $\sigma \in \Rm_+$, let us define
\begin{align*}
m^p_i(s;x,\sigma_i) & := \sum_{j=1}^p \frac 1{j!} \nabla^j f_i(x)[s]^j + \frac{\sigma_i}{p!} \|s\|^{p+1}, \quad i = 1,\ldots,m, \\
m^p(s;x,\sigma) & := \max_{i=1,\ldots,m}{m^p_i(s;x,\sigma_i)},
\end{align*}
so that Problem~\eqref{high_order_model} can be rewritten as
\[
\min_{s \in \R^n}{m^p(s;x,\sigma)}.
\]

We see that each $m^p_i(s;x,\sigma_i)$ is obtained by computing the Taylor series $T^p_i(x,s)$ of the function $f_i(x+s)$ at $x$ truncated at order $p$, then subtracting $f_i(x)$ (which is constant with respect to $s$) and adding a regularization term of order $p+1$. Namely,
\[
m^p_i(s;x,\sigma_i) = T^p_i(x,s) - f_i(x) + \frac{\sigma_i}{p!} \|s\|^{p+1}, \quad i = 1,\ldots,m,
\]
where
\[
T^p_i(x,s) := f_i(x) + \sum_{j=1}^p \frac 1{j!} \nabla^j f_i(x)[s]^j, \quad i = 1,\ldots,m.
\]

Note that, when $p=1$ in~\eqref{high_order_model}, we recover the strongly convex Problem~\eqref{1st_order_model} (up to a constant in the regularization terms). However, when $p>1$, Problem~\eqref{high_order_model} is no longer convex.
In the next result, we highlight that the optimal value of $m^p(;x,\sigma)$ with respect to $s$ cannot be positive.
\begin{proposition}\label{prop:sk}
Given $x \in \Rn$ and $\sigma \in \Rm_+$, let $\bar s$ be an optimal solution of Problem~\eqref{high_order_model}. Then,
\[
m^p(\bar s;x,\sigma) \le m^p(0;x,\sigma) = 0,
\]
that is, $m^p_i(\bar s;x,\sigma_i) \le m^p_i(0;x,\sigma_i) = 0$ for all $i = 1,\ldots,m$.
\end{proposition}

\begin{proof}
It follows straightforwardly from the definition of Problem~\eqref{high_order_model}.
\end{proof}

Reasoning similarly as in Section~\ref{sec:quad_reg_model}, we can give an equivalent reformulation of Problem~\eqref{high_order_model} as
\begin{equation}\label{high_order_model_eq}
\begin{split}
& \min_{(s,t) \in \R^{n+1}}{t} \\
\text{s.t. } & m^p_i(s;x,\sigma_i) \le t, \quad i = 1,\ldots,m.
\end{split}
\end{equation}
For Problem~\eqref{high_order_model_eq}, we define the Lagrangian function as
\[
L(s,t,\lambda) = t + \sum_{i=1}^m \lambda_i (m^p_i(s;x,\sigma_i) - t),
\]
so that the KKT conditions can be expressed as follows:
\begin{subequations}
\begin{align}
& \sum_{i=1}^m \lambda_i \nabla_s m^p_i(s;x,\sigma_i) = 0, \label{kkt1} \\
& \sum_{i=1}^m \lambda_i = 1, \label{kkt2} \\
& \lambda_i (m^p_i(s;x,\sigma_i) - t) = 0, \quad i = 1,\ldots,m, \label{kkt3} \\
& m^p_i(s;x,\sigma_i) - t \le 0, \quad i = 1,\ldots,m, \label{kkt4} \\
& \lambda_i \ge 0, \quad i = 1,\ldots,m,  \label{kkt5}
\end{align}
\end{subequations}
where $\nabla_s m^p_i(s;x,\sigma_i)$ denotes the gradient of $m^p_i(s;x,\sigma_i)$ with respect to $s$.

In the following result, we point out that a constraint qualification holds for Problem~\eqref{high_order_model_eq}, so that the KKT conditions~\eqref{kkt1}--\eqref{kkt5} are indeed necessary for optimality.

\begin{proposition}\label{prop:kkt}
Given $x\in \Rn$ and $\sigma \in \Rm_+$, let $(s,t)$ be an optimal solution of Problem \eqref{high_order_model_eq}. Then, there exists $\lambda \in \Rm$ such that KKT conditions~\eqref{kkt1}--\eqref{kkt5} are satisfied.
\end{proposition}

\begin{proof}
Since $(s,t)$ is an optimal solution for Problem~\eqref{high_order_model_eq}, using the Fritz-John conditions~\cite{bertsekas:1999}, there exist $\lambda_0,\lambda_1,\ldots,\lambda_m \in \R$ such that
\begin{subequations}
\begin{align}
& \sum_{i=1}^m \lambda_i \nabla_s m^p_i(s;x,\sigma_i) = 0, \label{fj1} \\
& \sum_{i=1}^m \lambda_i = \lambda_0, \label{fj2} \\
& \lambda_i (m^p_i(s;x,\sigma_i) - t) = 0, \quad i = 1,\ldots,m, \label{fj3} \\
& m^p_i(s;x,\sigma_i) - t \le 0, \quad i = 1,\ldots,m, \label{fj4} \\
& \lambda_i \ge 0, \quad i = 0,1,\ldots,m,  \label{fj5} \\
& (\lambda_0,\lambda_1,\ldots,\lambda_m) \ne (0,0,\ldots,0). \label{fj6}
\end{align}
\end{subequations}
To prove the result, we need to show that $\lambda_0 > 0$ (so that, by dividing each $\lambda_i$ by $\lambda_0$, the KKT conditions~\eqref{kkt1}--\eqref{kkt5} hold).
Reasoning by contradiction, assume that $\lambda_0 = 0$.
In view of~\eqref{fj2} and~\eqref{fj5}, it follows that
$\lambda_i = 0$ for all $i = 0,1,\ldots,m$, thus contradicting~\eqref{fj6}.
\end{proof}

Let us finally state useful properties of $f_1,\ldots,f_m$
coming from known results on functions with Lipschitz continuous $p$th-order derivatives~\cite{birgin:2017,cartis:2018}:
for all $i = 1,\ldots,m$, it holds that
\begin{gather}
|f_i(x+s) - T_i^p(x,s)| \le \frac{L_i}{p!} \|s\|^{p+1}, \quad \forall x,s \in \Rn, \label{ub_lips} \\
\|\nabla f_i(x+s) - \nabla_s T_i^p(x,s)\| \le \frac{L_i}{(p-1)!} \|s\|^p, \quad \forall x,s \in \Rn. \label{ineq_lips}
\end{gather}
Moreover, in the remainder of the paper, we will denote the $m$-dimensional vector
\[
L := \begin{bmatrix} L_1 & \ldots & L_m \end{bmatrix}^T
\]
and use the value $\Lmax$ defined as follows:
\begin{equation}\label{Lmax}
\Lmax := \max_{i=1,\ldots,m}{L_i}.
\end{equation}

\section{HOP: a High-Order algorithm for Pareto-front reconstruction}\label{sec:HOP}
In this section, we describe the proposed \textit{High-Order algorithm for Pareto-front reconstruction}, that is, Algorithm HOP.
At each iteration $k$, we are given a set $X_k$ of mutually nondominated points and we try to update $X_k$ by means of the \textit{Regularized Search} (RS) procedure.
Specifically, for any non-Pareto-stationary point $x_k$ in $X_k$, the RS procedure generates several trial points by computing (possibly approximate) minimizers of the regularized model $m^p(s;x_k,\sigma)$ with respect to $s$ using different choices for the regularization parameters $\sigma$. More precisely, we execute a while loop where we repeatedly increase the regularization parameters associated with the objective functions that do not satisfy a sufficient decrease condition in the trial points. The while loop stops only when all the objective functions satisfy the sufficient decrease condition, then returning $s(x_k)$ such that $x_k+s(x_k) \prec x_k$.
Notably, the trial points produced in the while loop are not necessarily discarded.
Instead, any trial point that makes at least one objective function decrease is considered a candidate to be included in the set $X_{k+1}$ for the next iteration.

\begin{algorithm}[ht!]
\caption{High-Order algorithm for Pareto-front reconstruction (HOP)}
\begin{algorithmic}[1]
\State {\bf given} $X_0 \subseteq \R^n$
\State set $k=0$
\While{$X_k$ is not Pareto-stationary}
\State set $\tilde X_k = X_k$
\ForAll{$x_k\in X_k$}
\If{$x_k$ is not Pareto-stationary}
\State compute $Y(x_k)$, $s(x_k)$, $\sigma(x_k)$ by \texttt{Regularized Search}$(x_k,X_k)$ 
\State set $\tilde X_k \leftarrow \tilde X_k\cup  Y(x_k)$
\EndIf
\EndFor
\State set $X_{k+1}$ as the nondominated points in  $\tilde X_k$
\State set $k\leftarrow k+1$
\EndWhile
\end{algorithmic}
\end{algorithm}

\begin{algorithm}[ht!]
\caption{\texttt{Regularized Search}$(x,X)$}
\begin{algorithmic}[1]
\Statex {{\bf input:} $x,X$;} \ {\bf output:} $Y,s,\sigma$
\State {\bf given} $\eta \in (0,1)$, $\delta \in (0,1)$ and $0<(\sigma_l)_i\leq(\sigma_u)_i<\infty$, $i=1,\dots,m$
\State set $j=0$, $Y^j=X$ and choose $\sigma^j\in[\sigma_l,\sigma_u]$
\State compute a (possibly approximate) minimizer $s^j$ of $m^p(s;x,\sigma^j)$
\While{$F(x +s^j) \not \le F(x) -\eta \,\dfrac{\|s^j\|^{p+1}}{p!}\sigma^j$}
\If{$F(x +s^j) \not > F(y) -\eta \,\dfrac{\|s^j\|^{p+1}}{p!}\sigma^j \text{ for all } y\in Y^j$} \label{RS_test}
\State set $Y^{j+1} = Y^j\cup\{x+s^j\}$
\Else
\State set $Y^{j+1} = Y^j$
\EndIf
\State set $\sigma^{j+1}$ such that, for all $i = 1,\ldots,m$,
        \[
        \displaystyle{(\sigma^{j+1})_i = \begin{cases}(\sigma^j)_i/ \delta \quad & \text{if }  f_i(x+s^j)> f_i(x) -\eta\dfrac{\|s^j\|^{p+1}}{p!}(\sigma^j)_i \\
        (\sigma^j)_i \quad & \text{otherwise}
        \end{cases}}
        \]
\State compute a (possibly approximate) minimizer $s^{j+1}$ of $m^p(s;x,\sigma^{j+1})$
\State set $j\leftarrow j+1$
\EndWhile
\State set $Y^{j+1} = Y^j\cup\{x+s^j\}$
\State {\bf return} $Y=Y^{j+1}$, $s =s^j$ and $\sigma=\sigma^j$
\end{algorithmic}
\end{algorithm}

Observe that RS procedure initially chooses the regularization parameter vector in $[\sigma_l, \sigma_u]$, with $\sigma_l$ and $\sigma_u$ being vectors of finite positive real numbers. In the remainder of the paper, we denote 
\begin{equation*}
\begin{split}
    \sigma^{\max}_u & := \max_{i=1,\dots,m}(\sigma_u)_i,\\
    \sigma^{\min}_l & := \min_{i=1,\dots,m}(\sigma_l)_i.
\end{split}
\end{equation*}

\subsection{Using exact minimizers of the regularized model}
First, we assume for simplicity that, for any $\sigma \in \Rm_+$, each $s^j$ computed in the RS procedure is a global minimizer of the regularized model $m^p(s;x,\sigma^j)$ with respect to $s$, that is,
\begin{equation}\label{s_global_min}
s^j\in\argmin_{s\in\R^n}m^p(s;x,\sigma^j).
\end{equation}
We now show how to relate the norm of an optimal solution $s$ of $\min_{s\in\R^n} m^p(s;x,\sigma)$ to $\|v(x+s)\|$, the latter measuring the amount of first-order stationarity violation at $x+s$ according to Proposition~\ref{prop:stat_1st_order}.

\begin{proposition}\label{prop:g_ub}
Given $x \in \Rn$ and $\sigma \in \Rm_+$, let $\bar s \in \argmin_{s\in\R^n} m^p(s;x,\sigma)$. Then,
\[
\|v(x+\bar s)\| \le \biggl(\frac{(p+1)\max_{i=1,\ldots,m} \sigma_i}{p!} + \frac{\Lmax}{(p-1)!}\biggr) \|\bar s\|^p,
\]
where $v(\cdot)$ is defined as in~\eqref{vx}.
\end{proposition}

\begin{proof}
Let $t := m^p(\bar s;x,\sigma)$. Since $\bar s \in \argmin_{s\in\R^n} m^p(s;x,\sigma)$, it follows that $(\bar s,t)$ is an optimal solution of Problem~\eqref{high_order_model_eq}. Using Proposition~\ref{prop:kkt},
there exists a KKT multiplier vector $\lambda \in \Rm$ satisfying~\eqref{kkt1}--\eqref{kkt5} at $(\bar s,t)$.
Using Lemma~\ref{lemma:dual} and the triangle inequality, we have that
\begin{equation}\label{g_ub}
\begin{split}
\|v(x+\bar s)\| & \le \biggl\|\sum_{i=1}^m \lambda_i \nabla f_i(x+\bar s)\biggr\| \\
& = \biggl\|\sum_{i=1}^m \lambda_i (\nabla_s T_i^p(x,\bar s) + \nabla f_i(x+\bar s) - \nabla_s T_i^p(x,\bar s)) \biggr\| \\
& \le \biggl\|\sum_{i=1}^m \lambda_i \nabla_s T_i^p(x,\bar s)\biggr\| + \biggl\|\sum_{i=1}^m \lambda_i (\nabla f_i(x+\bar s) - \nabla_s T_i^p(x,\bar s)) \biggr\|.
\end{split}
\end{equation}
To upper bound the first norm in the last term of~\eqref{g_ub}, using the triangle inequality we can write
\begin{equation}\label{g_ub1}
\begin{split}
& \biggl\|\sum_{i=1}^m \lambda_i \nabla_s T_i^p(x,\bar s)\biggr\| =\\
& \biggl\|\sum_{i=1}^m \lambda_i \Bigr(\nabla_s T_i^p(x,\bar s) + \frac{(p+1)\sigma_i}{p!} \|\bar s\|^{p-1} \bar s - \frac{(p+1)\sigma_i}{p!} \|\bar s\|^{p-1} \bar s\Bigr)\biggr\| \le \\
& \biggl\|\sum_{i=1}^m \lambda_i \Bigl(\nabla_s T_i^p(x,\bar s) + \frac{(p+1)\sigma_i}{p!} \|\bar s\|^{p-1} \bar s\Bigr)\biggr\| + \sum_{i=1}^m \lambda_i \frac{(p+1)\sigma_i}{p!} \|\bar s\|^p = \\
& \biggl\|\sum_{i=1}^m \lambda_i(\nabla_s m_i^p(\bar s;x,\sigma_i)) \biggr\| + \sum_{i=1}^m \lambda_i \frac{(p+1)\sigma_i}{p!} \|\bar s\|^p \le \\
& \frac{(p+1)\max_{i=1,\ldots,m}\sigma_i}{p!} \|\bar s\|^p,
\end{split}
\end{equation}
where we have used~\eqref{kkt5} in the first inequality, while the last inequality follows from~\eqref{kkt1} and~\eqref{kkt2}.
To upper bound the second norm in the last term of~\eqref{g_ub}, using the triangle inequality we can write
\begin{equation}\label{g_ub2}
\begin{split}
\biggl\|\sum_{i=1}^m \lambda_i (\nabla f_i(x+\bar s) - \nabla_s T_i^p(x,\bar s)) \biggr\| & \le
\sum_{i=1}^m \lambda_i \Bigl\|\nabla f_i(x+\bar s) - \nabla_s T_i^p(x,\bar s)\Bigr\| \\
& \le \sum_{i=1}^m \lambda_i \frac{L_i}{(p-1)!} \|\bar s\|^p \\
& \le \frac{\Lmax}{(p-1)!} \|\bar s\|^p,
\end{split}
\end{equation}
where, in the first two inequalities, we have used~\eqref{kkt5} and~\eqref{ineq_lips}, respectively,
while the third inequality follows from~\eqref{kkt2} and the definition of $\Lmax$ given in~\eqref{Lmax}.
Combining~\eqref{g_ub}, \eqref{g_ub1} and~\eqref{g_ub2}, the desired result follows.
\end{proof}

\begin{remark}\label{rem:stats0}
Proposition~\ref{prop:g_ub} guarantees that, for any $\sigma \in \Rm_+$, a point $x\in\R^n$ is Pareto-stationary for Problem~\eqref{prob} if $0 \in \argmin_{s\in\R^n} m^p(s;x,\sigma)$ (see item~\ref{prop:stat_1st_order_pareto} of Proposition~\ref{prop:stat_1st_order}).
\end{remark}

The following Propositions~\ref{prop:sigma_ub}--\ref{prop:F_decr} show that the while loop in the RS procedure ends after solving a finite number of regularized models. In particular, for each non-Pareto-stationary point $x_k\in X_k$, we get $s(x_k)$ such that $x_k+s(x_k) \prec x_k$.

\begin{proposition}\label{prop:sigma_ub}
Given $x\in \R^n$, assume that $\sigma_i \ge (1-\eta)^{-1} L_i$, $i=1,\ldots,m$, with $\eta \in (0,1)$, and $\bar s \in \argmin_{s \in \Rn} m^p(s;x,\sigma)$. Then,
\[
F(x+\bar s) \le F(x) -\eta \,\dfrac{\|\bar s\|^{p+1}}{p!}\sigma.
\]
\end{proposition}

\begin{proof}
Without loss of generality, assume that $\bar s \ne 0$ (otherwise, $x$ would be Pareto-stationary for Problem~\eqref{prob}, according to Remark~\ref{rem:stats0}, and the result would trivially follow).
Consider any objective index $i \in \{1,\ldots,m\}$.
Define
\begin{equation}\label{rho}
\rho_i := \frac{f_i(x) - f_i(x+\bar s)}{T_i^p(x,0) - T_i^p(x,\bar s)}.
\end{equation}
Observe that
\[
T_i^p(x,0) = f_i(x)
\]
and
\begin{equation}\label{num_rho}
T_i^p(x,0) - T_i^p(x,\bar s) = f_i(x) - T_i^p(x,\bar s) = -m_i^p(\bar s;x,\sigma_i) + \frac{\sigma_i}{p!} \|\bar s\|^{p+1} \ge \frac{\sigma_i}{p!} \|\bar s\|^{p+1},
\end{equation}
where the last inequality follows from Proposition~\ref{prop:sk}.
Hence, we get
\[
1- \rho_i = \frac{f_i(x+\bar s) - T_i^p(x,\bar s)}{-m_i^p(\bar s;x,\sigma_i) + (\sigma_i/p!) \|\bar s\|^{p+1}} \le \frac{f_i(x+\bar s) - T_i^p(x,\bar s)}{(\sigma_i/p!) \|\bar s\|^{p+1}}.
\]
In the last term, the numerator can be upper bounded by $(L_i/p!) \|\bar s\|^{p+1}$ using~\eqref{ub_lips}, leading to
\[
1- \rho_i \le \frac{L_i}{\sigma_i},
\]
Since, by hypothesis, we have $\sigma_i \ge L_i/(1-\eta)$, with $1-\eta>0$, it follows that $1-\rho_i \le 1-\eta$, that is, $\rho_i \ge \eta$.
Namely, recalling~\eqref{rho} and~\eqref{num_rho}, we have that
\begin{equation*}
\begin{split}
f_i(x) - f_i(x+\bar s) \ge \eta \frac{\sigma_i}{p!} \|\bar s\|^{p+1},
\end{split}
\end{equation*}
where the last inequality follows from Proposition~\ref{prop:sk} and the fact that $\eta > 0$.
\end{proof}

\begin{proposition}\label{prop:F_decr}
Assume that each $s^j$ in the RS procedure satisfies~\eqref{s_global_min}.
Given a non-Pareto-stationary point $x_k \in X_k$ generated at iteration $k$ of Algorithm HOP, the RS procedure computes $\sigma(x_k) \le \max\{(\delta(1-\eta))^{-1} L, \sigma_u\}$ such that
\[
F(x_k+s(x_k)) \le F(x_k) -\eta \,\dfrac{\|s(x_k)\|^{p+1}}{p!}\sigma(x_k).
\]
\end{proposition}

\begin{proof}
It follows from Proposition~\ref{prop:sigma_ub} and the instructions of the RS procedure.
\end{proof}

\begin{remark}\label{rem:sigma}
For any iteration $k$ and non-Pareto-stationary point $x_k \in X_k$, it follows from Proposition~\ref{prop:F_decr} and the instructions of the RS procedure that
\begin{equation}\label{sigmamax}
(\sigma(x_k))_i \le \sigma^{\max} := \max\{(\delta(1-\eta))^{-1} \Lmax, \sigma^{\max}_u\}, \quad i=1,\ldots,m.
\end{equation}
\end{remark}

\begin{remark}\label{rem:nf}
The maximum number of function evaluations $n_F$
required by the RS procedure at every iteration is upper bounded by $1$ plus the maximum number of times we increase the regularization parameter for each objective function $f_i$, $i=1,\ldots,m$. Hence, if each $s^j$ in the RS procedure satisfies~\eqref{s_global_min}, by Proposition~\ref{prop:F_decr} we have
\begin{equation}\label{nF}
n_F := \max_{i=1,\ldots,m} \biggl\lceil \frac 1{\log\delta}\log\biggl(\frac{(\sigma_l)_i}{\max\{(\delta(1-\eta))^{-1}L_i,(\sigma_u)_i\}}\biggr) \biggr\rceil + 1.
\end{equation}
\end{remark}

Now we show that, at every iteration $k$, there is no point $x\in X_k$ such that $x\prec x_k+s(x_k)$, that is, new nondominated points can be added to the current set $X_k$.

\begin{proposition}\label{prop:dirunc}
Assume that each $s^j$ in the RS procedure satisfies~\eqref{s_global_min}.
Given a non-Pareto-stationary point $x_k \in X_k$
generated at iteration $k$ of Algorithm HOP, we have that
\[
F(x_k + s(x_k))\not > F(y)-\eta \dfrac{\|s(x_k)\|^{p+1}}{p!}\sigma(x_k), \quad \forall y\in X_k.
\]
\end{proposition}
\begin{proof}
Since $x_k \in X_k$ is non-Pareto-stationary for Problem~\eqref{prob}, by Proposition~\ref{prop:F_decr}
we can write
\begin{equation}\label{eq:xk}
  f_{i}(x_k + s(x_k)) \le f_{i}(x_k) -\eta \dfrac{\|s(x_k)\|^{p+1}}{p!}(\sigma(x_k))_{i}, \quad\forall i=1,\dots,m.
\end{equation}
Since, by the instructions of the algorithm, there is no point in $X_k\setminus\{x_k\}$ that do\-mi\-na\-tes $x_k$, then, for any given $y\in X_k\setminus\{x_k\}$,
an index ${\hat \imath}\in \{1,\dots,m\}$ exists such that
 \[
 f_{\hat \imath}(y) > f_{\hat \imath}(x_k),
 \]
which, by~\eqref{eq:xk}, leads to
\begin{equation}\label{fy_proof2}
f_{\hat \imath}(x_k + s(x_k)) < f_{\hat \imath}(y) -\eta \dfrac{\|s(x_k)\|^{p+1}}{p!}(\sigma(x_k))_{\hat\imath},\quad\forall y\in X_k\textit{}\setminus \{ x_k\}.
\end{equation}
The desired result is thus obtained by combining~\eqref{eq:xk} and~\eqref{fy_proof2}.
\end{proof}

Using Propositions~\ref{prop:g_ub} and~\ref{prop:F_decr}, we now show that, whenever a new point is added to the set $Y^j(x_k)$ during the while loop in the RS procedure,
the Pareto-stationarity violation can be related to the decrease obtained in one objective function. 
Furthermore, by Proposition~\ref{HIincrease_gen}, this decrease corresponds to an increase in the hypervolume of the set $F(Y^j(x_k))$.
This is formalized in the following result.

\begin{proposition}\label{prop:HI_increase}
Assume that each $s^j$ in the RS procedure satisfies~\eqref{s_global_min}.
Given a non-Pareto-stationary point $x_k \in X_k$ generated at iteration $k$ of Algorithm HOP, consider the $j$th iteration of the RS procedure invoked with $x=x_k$ and $X=X_k$. Let $\sigma^j(x_k) := \sigma^j$, $s^j(x_k) := s^j$ and
$Y^j(x_k) := Y^j$. If $Y^{j+1}(x_k) \neq Y^j(x_k)$,
using $v(\cdot)$  defined as in~\eqref{vx}, then
\begin{enumerate}[label=(\roman*)]
\item an index $\displaystyle{\hat \imath \in \{1,\ldots,m\}}$  exists such that 
\[\displaystyle{f_{\hat \imath}(x_k) - f_{\hat \imath}(x_k + s^j(x_k)) \ge c \|v(x_k + s^j(x_k))\|^{(p+1)/p}},\] \label{f_decrease}
\item $\displaystyle{HI(F(Y^{j+1}(x_k))) - HI(F(Y^j(x_k))) \ge c^m\|v(x_k+s^j(x_k))\|^{m(p+1)/p}}$,\label{HI_increase}
\end{enumerate}

where
\begin{equation}\label{c}
c := \Biggl(\eta \,\dfrac{\sigma_l^{\min}}{p!}\Biggr) \biggl(\frac{(p+1)\sigma^{\max}}{p!} + \frac{\Lmax}{(p-1)!}\biggr)^{-(p+1)/p}
\end{equation}
and $\sigma^{\max}$ is defined as in~\eqref{sigmamax}.
\end{proposition}

\begin{proof}
Since $Y^{j+1}(x_k) \neq Y^j(x_k)$, from line~\ref{RS_test} of the RS procedure (using $y=x_k$), we have that $\hat \imath \in \{1,\ldots,m\}$  exists such that
\begin{equation}\label{H_lb_s}
f_{\hat \imath}(x_k) - f_{\hat \imath}(x_k + s^j(x_k)) \ge \eta \,\dfrac{\|s^j(x_k)\|^{p+1}}{p!}\sigma^j_{\hat \imath}(x_k) \ge \eta \,\dfrac{\sigma_l^{\min}}{p!} \|s^j(x_k)\|^{p+1}.
\end{equation}
Moreover, from Proposition~\ref{prop:g_ub} and Remark~\ref{rem:sigma}, we can write
\begin{equation*}
\begin{split}
\|v(x_k+s^j(x_k))\| & \le \biggl(\frac{(p+1)\max_{i=1,\ldots,m} (\sigma^j(x_k))_i}{p!} + \frac{\Lmax}{(p-1)!}\biggr) \|s^j(x_k)\|^{p} \\
 & = \biggl(\frac{(p+1)\max_{i=1,\ldots,m} (\sigma(x_k))_i}{p!} + \frac{\Lmax}{(p-1)!}\biggr) \|s^j(x_k)\|^{p} \\
& \le \biggl(\frac{(p+1)\sigma^{\max}}{p!} + \frac{\Lmax}{(p-1)!}\biggr) \|s^j(x_k)\|^{p},
\end{split}
\end{equation*}
that is,
\begin{equation*}
\|s^j(x_k)\|^{(p+1)} \ge
\biggl(\frac{(p+1)\sigma^{\max}}{p!} + \frac{\Lmax}{(p-1)!}\biggr)^{-(p+1)/p} \|v(x_k+s^j(x_k))\|^{(p+1)/p}.
\end{equation*}
Hence, item~\ref{f_decrease} follows by combining the above inequality with~\eqref{H_lb_s}, while item~\ref{HI_increase} follows from item~\ref{f_decrease} and the fact that, from Lemma~\ref{HIincrease_gen}, we have
\[
HI(F(Y^{j+1}(x_k))) - HI(F(Y^j(x_k))) \ge \left(f_{\hat \imath}(x_k) - f_{\hat \imath}(x_k + s^j(x_k))\right)^m. 
\]
\end{proof}

\begin{remark}\label{rem:HI_increase}
For any $k \ge 0$ and any non-Pareto-stationary point $x_k \in X_k$, from Propositions~\ref{prop:dirunc}--\ref{prop:HI_increase} the following holds:
\begin{enumerate}[label=(\roman*)]
\item $\displaystyle{f_i(x_k) - f_i(x_k + s(x_k)) \ge c \|v(x_k + s(x_k))\|^{(p+1)/p}}$ for all $i = 1,\ldots,m$, \label{f_decrease_x}
    \item $\displaystyle{HI(F(X_{k+1})) - HI(F(X_k)) \ge c^m(\|v(x_k+s(x_k))\|)^{m(p+1)/p}}$. \label{HI_increase_x}
\end{enumerate}
In particular, $x_k + s(x_k)$ ensures a decrease for all objective functions since it is the point produced in the last iteration of the RS procedure invoked with $x=x_k$ and $X=X_k$.
\end{remark}

We are now ready to derive worst-case complexity bounds for Algorithm HOP in two different scenarios.
Namely, we are interested in upper bounding the number of iterations and function evaluations needed to generate a set $X_k$ satisfying one of the following two properties:
\begin{itemize}
\item all points in $X_k$ are $\epsilon$-approximate Pareto-stationary, as shown in Theorem~\ref{th:Keps1};
\item  at least one point $x_k \in X_k$ is such that $x_k + s(x_k)$ is an $\epsilon$-approximate Pareto-stationary point, as shown in Theorem~\ref{th:Keps2}.
\end{itemize}

\begin{theorem}\label{th:Keps1}
Assume that each $s^j$ in the RS procedure satisfies~\eqref{s_global_min}.
Let $\eps>0$ and, for Algorithm HOP, define
\[
\Keps' := \{k \colon \|v(x_k+s(x_k))\| \ge \eps\ \text{for at least one}\ x_k\in X_k\}.
\]
Let $NF_{\eps}'$ be the number of functions evaluations performed by Algorithm HOP up to the first iteration $\bar k \notin K_{\epsilon}'$.
If Assumption~\ref{ass:boundedness} holds, then
\begin{equation*}
\begin{split}
|K_{\eps}'| & \le \left\lfloor\dfrac{\overline{HI}-HI_0}{c^m}\eps^{-m(p+1)/p}\right\rfloor,\\
NF_{\eps}' & \le n_F |X(\eps)| \left\lfloor\dfrac{\overline{HI}-HI_0}{c^m}\eps^{-m(p+1)/p}\right\rfloor,
\end{split}
\end{equation*}
where $\overline{HI}$, $n_F$ and $c$ are given in~\eqref{HI}, \eqref{nF} and~\eqref{c}, respectively,
while $|X(\eps)| := \max_{k = 0,\ldots, \bar k}|X_k|$ and $HI_0 := HI(F(X_0))$.
\end{theorem}

\begin{proof}
Let us consider an iteration $r$ and write
\begin{equation*}
\begin{split}
HI(F(X_{r+1})) - HI_0 & = \sum_{k=0}^r (HI(F(X_{k+1})) - HI(F(X_k))).
\end{split}
\end{equation*}
Taking the limit for $r\to\infty$ and using item~\ref{HI_increase_x} in Remark~\ref{rem:HI_increase}, for any non-Pareto-stationary $x_k \in X_k$ we can write
\begin{equation*}
\begin{split}
\overline{HI} - HI_0 & \ge \sum_{k=0}^{\infty} (HI(F(X_{k+1})) - HI(F(X_k))) \\
& \ge \sum_{k=0}^{\infty} c^m \|v(x_k+s(x_k))\|^{m(p+1)/p} \\
& \ge \sum_{k\in \Keps'} c^m \|v(x_k+s(x_k))\|^{m(p+1)/p} \\
& \ge \sum_{k\in \Keps'}c^m\eps^{m(p+1)/p} \\
& = |\Keps'| c^m\eps^{m(p+1)/p}, 
\end{split}
\end{equation*}
which proves the upper bound on $|K_{\epsilon}'|$.

To upper bound $NF_\eps'$, first note that $\bar k \le |K_\eps'|$. 
Furthermore, as pointed out in Remark~\ref{rem:nf}, the maximum number of function evaluations $n_F$ required by the RS procedure at each iteration is given as in~\eqref{nF}.
Then, the desired result follows by observing that, at any iteration $k$, the RS procedure is called at most $|X_k|$ times.
\end{proof}

\begin{theorem}\label{th:Keps2}
Assume that each $s^j$ in the RS procedure satisfies~\eqref{s_global_min}.
Let $\eps>0$ and, for Algorithm HOP, define
\[
\Keps'' := \{k \colon \|v(x_k+s(x_k))\| \ge \eps\ \text{for all}\ x_k\in X_k\}.
\] 
Let $NF_{\eps}''$ be the number of functions evaluations performed by Algorithm HOP up to the first iteration $\bar k \notin K_{\epsilon}''$.
If Assumption~\ref{ass:boundedness} holds, then
\begin{equation*}
\begin{split}
|\Keps''| &\le\left\lfloor \biggl(\frac{\min_{i=1,\ldots,m}(f_i(x_0) - \fmin_i)}c\biggr) \epsilon^{-(p+1)/p}\right\rfloor,\\
NF_{\eps}'' & \le n_F |X(\eps)| \left\lfloor \biggl(\frac{\min_{i=1,\ldots,m}(f_i(x_0) - \fmin_i)}c\biggr) \epsilon^{-(p+1)/p}\right\rfloor,
\end{split}
\end{equation*}
where $n_F$ and $c$ are given in~\eqref{nF} and~\eqref{c}, respectively, while $X(\eps)$ is defined as in Theorem~\ref{th:Keps1}.
\end{theorem}

\begin{proof}
Using item~\ref{f_decrease_x} of Remark~\ref{rem:HI_increase} and the instructions for computing the set $X_{k+1}$ at every iteration $k$, we can define a sequence $\{x_k\} \subseteq \Rn$ such that, starting from any $x_0\in X_0$, we have that $F(x_{k+1})\le F(x_k + s(x_k))$ and $x_{k+1} \in X_{k+1}$ for all $k \ge 0$. Moreover, still using item~\ref{f_decrease_x} of Remark~\ref{rem:HI_increase},  we have
\[
f_i(x_k) - f_i(x_{k+1}) \ge f_i(x_k) - f_i(x_k + s(x_k)) \ge c \|v(x_{k+1})\|^{(p+1)/p}, \quad i = 1,\ldots,m,
\]
where $v(\cdot)$ is defined as in~\eqref{vx}.
Now, let $\hat\imath\in\{1,\dots,m\}$ be such that 
\[
f_{\hat \imath}(x_0) - \fmin_{\hat \imath} = \min_{i=1,\dots,m}(f_i(x_0)-\fmin_i).
\]
Note that $f_{\hat \imath}(x_0) - \fmin_{\hat \imath} < \infty$ by Assumption~\ref{ass:boundedness}.
Then, for any $k \ge 0$, summing up all the objective decreases until iteration $k$, we can write
\begin{equation}\label{fdecr_rate}
f_{\hat \imath}(x_0) - \fmin_{\hat \imath} \ge
f_{\hat \imath}(x_0) - f_{\hat \imath}(x_{k+1}) =
\sum_{\ell=0}^k (f_{\hat \imath}(x_{\ell}) - f_{\hat \imath}(x_{\ell+1})) \ge
c \sum_{\ell=0}^k \|v(x_{\ell+1})\|^{(p+1)/p}.
\end{equation}
Letting $k \to \infty$ in~\eqref{fdecr_rate}, we obtain
\[
f_{\hat \imath}(x_0) - \fmin_{\hat \imath} \ge c \sum_{\ell=0}^{\infty} \|v(x_{\ell+1})\|^{(p+1)/p} \ge c \sum_{\ell \in \Keps''} \|v(x_k)\|^{(p+1)/p} \ge c |\Keps''| \epsilon^{(p+1)/p},
\]
hence proving the upper bound on $|K_{\epsilon}''|$. 

The upper bound on $NF_{\eps}''$ can be obtained by the same reasoning as in the proof of Theorem~\ref{th:Keps1}, noting that $\bar k \le |K_\eps''|$.
\end{proof}

\begin{remark}
Assumption~\ref{ass:boundedness} can be weakened in Theorem~\ref{th:Keps2}. Namely, Theorem~\ref{th:Keps2} still holds if we assume that $\hat \jmath \in \{1,\ldots,m\}$ exists such that $\fmin_{\hat\jmath}>-\infty$.
\end{remark}

From Theorem~\ref{th:Keps1} we see that, to produce a set $X_k$ where all points are $\epsilon$-approximate Pareto-stationary, the maximum number of iterations and function evaluations required by Algorithm HOP are $\bigO{\epsilon^{-m(p+1)/p}}$ and $\bigO{|X(\eps)|\epsilon^{-m(p+1)/p}}$, respectively, thus depending on both the number of objective functions $m$ and the model order $p$.
Interestingly, the dependency on $m$ does not appear in the worst-case complexity bounds to generate at least one $\epsilon$-approximate Pareto-stationary point, which are $\bigO{\epsilon^{-(p+1)/p}}$ for the number of iterations and $\bigO{|X(\eps)| \epsilon^{-(p+1)/p}}$ for the number of functions evaluations, as stated in Theorem~\ref{th:Keps2}.
Note the latter result aligns with the worst-case complexity bounds given in~\cite{calderon:2022} for algorithms that generate a single point.

\subsection{Using inexact minimizers}
In this section, we analyze Algorithm HOP assuming that $s$ is chosen as an approximate minimizer of $m^p(s;x,\sigma)$ for any point $x$ and regularization parameter vector $\sigma$.
Specifically, we will show that we maintain the same worst-case complexity as in the above subsection (up to constant factors).

Here, given a point $x \in \Rn$ and a regularization parameter vector $\sigma \in \Rm_+$, we just require $s$ and $\lambda_i$, $i=1,\dots,m$, to satisfy the following conditions:
\begin{subequations}\label{KKT_approx}
\begin{align}
& m^p(s;x,\sigma) \le m^p(0;x,\sigma) = 0, \label{opt0_approx} \\
& \sum_{i=1}^m \lambda_i \le \theta, \label{opt1_approx} \\
& \Bigl\|\sum_{i=1}^m \lambda_i \nabla_s m^p_i(s;x,\sigma_i)\Bigr\| \le \tau \|s\|^p, \label{opt2_approx}\\
& \lambda_i \ge 0, \quad i = 1,\ldots,m,  \label{opt3_approx}
\end{align}
\end{subequations}
with given $\theta \in [1,\infty)$ and $\tau \in [0,\infty)$.

We see that~\eqref{opt0_approx} imposes a decrease for $m^p(s;x,\sigma)$, while~\eqref{opt1_approx}--\eqref{opt3_approx} express approximate KKT conditions (cf. the original KKT conditions given in~\eqref{kkt1}--\eqref{kkt5}).
In practice, to obtain $s$ satisfying~\eqref{opt0_approx}--\eqref{opt3_approx}, we can perform an inexact minimization of $m^p(s;x,\sigma)$ with respect to $s$ and get $\lambda$ as an approximate KKT multiplier vector, thus avoiding the issues connected with global optimization.

Then, Propositions~\ref{prop:g_ub} and~\ref{prop:sigma_ub} can be reformulated according to the new approximate conditions, leading to the same worst-case complexity bounds (up to constant factors) by the same arguments used in the previous subsection.

\begin{proposition}\label{prop:g_ub_approx}
Given $x \in \Rn$ and $\sigma \in \Rm_+$, let $\bar s$ satisfying~\eqref{KKT_approx}. Then,
\[
\|v(x+\bar s)\| \le \biggl[\tau + \theta\biggl(\frac{(p+1)\max_{i=1,\ldots,m} \sigma_i}{p!} + \frac{\Lmax}{(p-1)!}\biggr)\biggr] \|\bar s\|^p,
\]
where $v(\cdot)$ is defined as in~\eqref{vx}.
\end{proposition}

\begin{proof}
We can reason as in the proof of Proposition~\ref{prop:g_ub}, with the only difference being in the fact that the last term in~\eqref{g_ub1}
must be multiplied by $\theta$, in view of~\eqref{opt1_approx}, and added to $\tau \|\bar s\|^p$, in view of~\eqref{opt2_approx},
while the last term in~\eqref{g_ub2} should be multiplied by $\theta$, in view of~\eqref{opt1_approx}.
\end{proof}

\begin{proposition}\label{prop:sigma_ub_approx}
Given a non-Pareto-stationary point $x\in \R^n$, assume that $\sigma_i \ge (1-\eta)^{-1} L_i$, $i=1,\ldots,m$, with $\eta \in (0,1)$, and $\bar s$ satisfying~\eqref{KKT_approx}. Then,
\[
F(x+\bar s) \le F(x) -\eta \,\dfrac{\|\bar s\|^{p+1}}{p!}\sigma.
\]
\end{proposition}

\begin{proof}
It is identical to the proof of Proposition~\ref{prop:sigma_ub} (just observing that~\eqref{opt0_approx} is now a condition on $s$
rather than coming from Proposition~\ref{prop:sk}).
\end{proof}

\begin{remark}\label{rem:compl_approx}
Replacing Propositions~\ref{prop:g_ub}--\ref{prop:sigma_ub} by Propositions~\ref{prop:g_ub_approx}--\ref{prop:sigma_ub_approx}, respectively, has only the effect to generate, instead of the constant $c$ defined in~\eqref{c}, the following new constant:
\begin{equation}\label{c_tilde}
\tilde c := \Biggl(\eta \,\dfrac{\sigma_l^{\min}}{p!}\Biggr) \biggl[\tau + \theta\biggl(\frac{(p+1)\sigma^{\max}}{p!} + \frac{\Lmax}{(p-1)!}\biggr)\biggr]^{-(p+1)/p}\hspace*{-1.3cm}.    
\end{equation}
Namely, if each $s^j$ in the RS procedure satisfies~\eqref{KKT_approx}, then the results of Theorems~\ref{th:Keps1}--\ref{th:Keps2} hold with $c$ replaced by $\tilde c$.
\end{remark}

\section{LHOP: a light version of HOP}\label{sec:LHOP}
Each iteration of Algorithm HOP may be computationally expensive as the RS procedure needs to be applied to every point in $X_k$.
In this section, we propose a variant called Light HOP (LHOP) where, at each iteration, the RS procedure is applied to a single point in the set $X_k$ rather than to all of them.
Note that, in the proposed version, we do not specify any rule to choose such a point, allowing any user-specified criterion to be applied.

\begin{algorithm}[ht!]
\caption{Light High-Order algorithm for Pareto-front reconstruction (LHOP)}
\begin{algorithmic}[1]
\State {\bf given} $X_0 \subseteq \R^n$
\State set $k=0$
\While{all $x_k\in X_k$ are non-Pareto-stationary}
\State choose $x_k\in X_k$
\State compute $Y(x_k)$, $s(x_k)$, $\sigma(x_k)$ by \texttt{Regularized Search}$(x_k,X_k)$ 
\State set $X_{k+1}$ as the nondominated points in  $X_k\cup Y(x_k)$
\State set $k\leftarrow k+1$
\EndWhile
\end{algorithmic}
\end{algorithm}

The worst-case complexity bounds for Algorithm LHOP are summarized in the next result.

\begin{theorem}\label{th:Keps2_LHOP}
Given $\eps>0$, define $\Keps''$ and $NF_{\eps}''$ as in Theorem~\ref{th:Keps2}.
If Assumption~\ref{ass:boundedness} holds, then
\begin{equation*}
\begin{split}
|K_{\eps}''| & \le \left\lfloor\dfrac{\overline{HI}-HI_0}{\hat c^m}\eps^{-m(p+1)/p}\right\rfloor,\\
NF_{\eps}'' & \le n_F \left\lfloor\dfrac{\overline{HI}-HI_0}{\hat c^m}\eps^{-m(p+1)/p}\right\rfloor,
\end{split}
\end{equation*}
with
\[
\hat c =
\begin{cases}
c \quad & \text{if each $s^j$ in the RS procedure satisfies~\eqref{s_global_min}}, \\
\tilde c \quad & \text{if each $s^j$ in the RS procedure satisfies~\eqref{KKT_approx}},
\end{cases}
\]
where $n_F$, $c$ and $\tilde c$ are given in~\eqref{nF}, \eqref{c} and~\eqref{c_tilde}, respectively.
\end{theorem}

\begin{proof}
If each $s^j$ in the RS procedure satisfies~\eqref{s_global_min}, then the proof is identical to the one of Theorem~\ref{th:Keps1} with $K_\eps'$ replaced with $K_\eps''$ since Algorithm LHOP  picks one point $x_k \in X_k$ at every iteration $k$, allowing us to state that
\[
HI(F(X_{{k}}\cup\{x_{{k}}+s(x_k)\})) - HI(F(X_{{k}})) \ge  c^m(\|v(x_k+s(x_k))\|)^{m(p+1)/p} \ge c^m\eps^{m(p+1)/p}
\]
 holds for all $k\in K_\eps''$ (but note that this is not true for all $k\in K_\eps'$).
 If each $s^j$ in the RS procedure satisfies~\eqref{KKT_approx}, then the reasoning is the same as before by taking into account Remark~\ref{rem:compl_approx}.
\end{proof}

For Algorithm LHOP, we can only provide worst-case complexity bounds for generating at least one $\epsilon$-approximate Pareto-stationary point. In particular, according to Theorem~\ref{th:Keps2_LHOP}, the worst-case complexity is $\bigO{\epsilon^{-m(p+1)/p}}$ for both the number of iterations and the number of function evaluations.
Hence, compared to the results given for HOP in Theorem~\ref{th:Keps2}, the maximum number of iterations for LHOP is worse since it depends on $m$, while the maximum number of function evaluations, although also dependent on $m$, does not depend on $|X(\eps)|$.
However, recall that LHOP performs a less exhaustive analysis of the points in the current set $X_k$ at each iteration $k$.

\section{Numerical examples}\label{sec:examples}
In this section, we report a preliminary numerical experience to compare the two algorithms HOP and LHOP and assess their behavior. By virtue of the results obtained, we will also try to put into perspective the different complexity results.

\paragraph{Problems used in the numerical experience} We use two classes of multiobjective problems, that is, problems that are inherently multiobjective (from the collections proposed in \cite{custodio2011direct} and \cite{Thomann19}) and multiobjective problems built using two functions selected from the CUTEst \cite{cutest} test set as proposed in \cite{de2026objective}.

\begin{table}
\centering
\small
    \begin{tabular}{rc|cccc}\hline
    Problem name & P.nr. &n & m & ref\\\hline
                  T1 & 1&  2 &  2 &\cite{custodio2011direct}\\
                  T2 & 2&  2 &  2 &\cite{custodio2011direct}\\
       LOVISON1 (unc) & 3&  2 &  2 &\cite{custodio2011direct}\\
       LOVISON3 (unc) & 4&  2 &  2 &\cite{custodio2011direct}\\
                FES2 & 5& 10 &  3 &\cite{custodio2011direct}\\
                IKK1 & 6&  2 &  3 &\cite{custodio2011direct}\\
                ZLT1 & 7&  4 &  3 &\cite{custodio2011direct}\\
BROWNDEN vs ALLINITU & 8&  4 &  2 &\cite{de2026objective}\\
  BROWNAL vs ARWHEAD & 9& 10 &  2 &\cite{de2026objective}\\
   BROWNAL vs VARDIM & 10& 10 &  2 &\cite{de2026objective}\\
   ARWHEAD vs VARDIM &11& 10 &  2 &\cite{de2026objective}\\
 ZANGWIL2 vs ROSENBR &12&  2 &  2 &\cite{de2026objective}\\
    ZANGWIL2 vs CUBE &13&  2 &  2 &\cite{de2026objective}\\
 ZANGWIL2 vs WAYSEA1 &14&  2 &  2 &\cite{de2026objective}\\
  ROSENBR vs WAYSEA1 &15&  2 &  2 &\cite{de2026objective}\\
     ROSENBR vs CUBE &16&  2 &  2 &\cite{de2026objective}\\
     WAYSEA1 vs CUBE &17&  2 &  2 &\cite{de2026objective}\\\hline
     \end{tabular}
    \caption{Description of the problems used in the numerical experience. Column labels are: `P.nr.' for problem number; `$n$' for the number of variables; `$m$' for the number of objective functions; `ref' for the reference where the problem definition can be found.}
    \label{tab:placeholder3}
\end{table}

\paragraph{Implementation details} Both the algorithms HOP and LHOP have been implemented in Matlab\textsuperscript{\textcopyright} (version 2026a) and tested on an intel\textsuperscript{\textregistered} Xeon\textsuperscript{\textregistered} w3-2425 with 12 cores and 12 GB RAM running Linux Ubuntu 24.04 LTS. In the implementation of HOP and LHOP, we made the following choices:
\begin{itemize}
\item \textit{Stationarity tolerance}. We use the value $\epsilon = 10^{-5}$;
\item \textit{Ordering of $X_k$}. Two strategies for ordering points in $X_k$ are considered, identified by the parameter \texttt{spread\_option} $\in \{1,2\}$:
\begin{description}
    \item \texttt{spread\_option} $=1$ means that, at every iteration $k$, we order $X_k$ on the basis of the \textit{spread} measure as defined in \cite{custodio:2021}; then, in LHOP, we select the point which is adjacent to the largest gap  in the current Pareto front in objective space;   
    \item \texttt{spread\_option} $=2$ means that, at every iteration $k$, we order $X_k$ on the basis of the stationarity measure of the points $x_k\in X_k$; then in LHOP we select the worst point in terms of stationarity measure;
\end{description}

\item \textit{Choice of regularized search parameters}. Ideally, $[\sigma_l, \sigma_u]$ should approximate~the~interval where the reciprocal of the (local) Lipschitz constants lie (see Proposition~\ref{prop:F_decr}).
In the absence of this information, one can draw inspiration from what is commonly done for regularized methods in single-objective optimization, which typically update the regularization parameter by means of rules inherited from trust-region schemes (see, e.g., \cite{birgin:2017,cartis:2011a}).
In our case, since we consider a set of points $X_k$ at each iteration $k$, we use a reference value given by the mean $\bar \sigma_k$ of the regularization parameters returned by the Regularized Search over all the points in $X_k$. Then, at iteration $k+1$, we initialize the Regularized Search with $\sigma = \bar \sigma_k/2$ for all points in $X_{k+1}$.
Note that this implicitly defines an interval $[\sigma_l, \sigma_u] \subset (0,\infty)^m$ containing the initial value of $\sigma$, since we start with a positive value for $\sigma$ at iteration $k=0$ (so that $\sigma_l$ is a positive vector) and, in view of Remark~\ref{rem:sigma}, the regularization parameters are upper bounded by finite values (so that $\sigma_u$ is a finite vector). 
Moreover, we use $\delta = 0.5$ in the Regularized Search, in order to double the value of each regularization parameter whenever a sufficient decrease is not obtained.
Note that since in LHOP we select a single point from $X_k$ at any iteration $k$, the mean vector $\bar \sigma_k$ coincides with $\sigma (x_k)$.

\end{itemize} 

\paragraph{Hypervolume computation} To compute the hypervolume, we first need to define the reference vector in objective space.
To this end, let us denote with ${\cal P}$ the set of multiobjective problems and with ${\cal S}$ the set of solvers in the comparison. For every multiobjective problem $p\in \cal P$ and solver $s\in\cal S$, let ${\cal F}_{p,s}$ be the Pareto front approximation found by solver $s$ for problem $p$. We define 
\[
{\cal F}_p = \bigcup_{s\in\cal S}{\cal F}_{p,s}.
\]
Then, we define the ideal $\tilde y^{I,p}$ and the nadir $\tilde y^{N,p}$ points for problem $p$ as follows:
\begin{equation*}
\begin{split}
\tilde y^{I,p} &=\left(\min_{y\in {\cal F}_p}y_1,\min_{y\in {\cal F}_p}y_2,\dots,\min_{y\in {\cal F}_p}y_{m_p}\right)^T,\\
\tilde y^{N,p} &=\left(\max_{y\in {\cal F}_p}y_1,\max_{y\in {\cal F}_p}y_2,\dots,\max_{y\in {\cal F}_p}y_{m_p}\right)^T+\mathbf{1},
\end{split}
\end{equation*}
where $m_p$ denotes the number of objective {functions} of problem $p$. According to Definition~\ref{HI_definition}, the hypervolume obtained by a solver $s\in\cal S$ is $HI({\cal F}_{p,s})$ where the reference vector is selected as $\rho = \tilde y^{N,p}$. 
Note that the hypervolume could be biased by the possible different scalings of the objective functions for problem $p$. To avoid privileging one objective function over the others in the computation of the hypervolume
  (as proposed in \cite{bigeon2021dmulti}), we define the following transformation $T:\R^{m_p}\to \R^{m_p}$ such that $T(y) = z$ with
\begin{equation}\label{bigeontransform}
z_i = \left\{\begin{array}{ll}
  (y_i- \tilde y^{I,p}_i)/(\tilde y^{N,p}_i-\tilde y^{I,p}_i) & \text{if}\ \tilde y^{N,p}_i-\tilde y^{I,p}_i\neq 0\\
  y_i- \tilde y^{I,p}_i & \text{otherwise}.
\end{array}\right.
\end{equation}
Then, given a problem $p\in\cal P$ and a solver $s\in\cal S$, the hypervolume $HV_{p,s}$ is computed as
	\[
	{
	HV_{p,s} = HI(T({\cal F}_{p,s})) = {\rm Vol}\left(\bigcup_{y\in{\cal F}_{p,s}}[T(y),T(\tilde y^{N,p})]\right)}.
	\]

\paragraph{Results on the test problems}
The results obtained by running the solvers HOP and LHOP on the 17 problems are reported in Tables~\ref{tab:placeholder_newhvconalfa_nocute}--\ref{tab:placeholder_newhvconalfa}. In particular, for each solver, we consider $p=1,2$, and ${\tt spread\_option}=1,2$.
\begin{table}[htb]
\centering
\small
    \begin{tabular}{l|cccr|cccr}\hline
    & \multicolumn{4}{c|}{HOP} & \multicolumn{4}{c}{LHOP}\\
    P.nr. & it & nP & nF & HV & it & nP & nF & HV \\\hline
    \multicolumn{9}{c}{\texttt{spread\_option}$=1$}\\\hline
1	&	17	&	4767	&	*	&	\bf	0.77	&	4973	&	4980	&	*	&		0.65	\\
2	&	6	&	4297	&	*	&		0.80	&	4978	&	60	&	*	&	\bf	0.82	\\
3	&	5	&	2468	&	*	&	\bf	0.72	&	4934	&	25	&	*	&		0.71	\\
4	&	7	&	769	&	*	&	\bf	0.21	&	4975	&	4967	&	*	&		0.00	\\
5	&	6	&	750	&	*	&	\bf	0.71	&	4957	&	65	&	*	&		0.32	\\
6	&	4	&	9179	&	*	&	\bf	0.98	&	4951	&	4998	&	*	&		0.97	\\
7	&	16	&	6469	&	*	&		$>$0.99	&	4979	&	4999	&	*	&		$>$0.99	\\
8	&	6	&	252	&	*	&	\bf	0.99	&	4925	&	27	&	*	&		0.91	\\
9	&	5	&	3131	&	*	&		0.90	&	4962	&	57	&	*	&	\bf	0.99	\\
10	&	26	&	2550	&	*	&		$>$0.99	&	4958	&	30	&	*	&		$>$0.99	\\
11	&	5	&	5444	&	*	&		0.83	&	4993	&	775	&	*	&	\bf	$>$0.99	\\
12	&	8	&	4189	&	*	&	\bf	0.95	&	4964	&	4986	&	*	&		0.91	\\
13	&	7	&	988	&	*	&	\bf	0.95	&	4967	&	4984	&	*	&		0.91	\\
14	&	4	&	1456	&	*	&	\bf	0.82	&	4946	&	92	&	*	&		0.81	\\
15	&	8	&	5216	&	*	&		0.98	&	4349	&	1111	&	*	&		0.98	\\
16	&	29	&	4896	&	*	&		$>$0.99	&	4959	&	4882	&	*	&		$>$0.99	\\
17	&	26	&	2777	&	*	&		$>$0.99	&	4747	&	4245	&	*	&		$>$0.99	\\\hline
    \multicolumn{9}{c}{\texttt{spread\_option}$=2$}\\\hline
1	&	16	&	1976	&	*	&	\bf	0.76	&	4878	&	312	&	*	&		0.72	\\
2	&	6	&	4713	&	*	&		0.80	&	4980	&	676	&	*	&	\bf	0.82	\\
3	&	4	&	5341	&	*	&		0.71	&	4912	&	93	&	*	&		0.71	\\
4	&	7	&	1573	&	*	&	\bf	0.12	&	4959	&	60	&	*	&		0.10	\\
5	&	6	&	719	&	*	&	\bf	0.68	&	4338	&	1704	&	*	&		0.48	\\
6	&	4	&	9192	&	*	&	\bf	0.98	&	4791	&	636	&	*	&		0.97	\\
7	&	16	&	6460	&	*	&		$>$0.99	&	4979	&	65	&	*	&		$>$0.99	\\
8	&	8	&	147	&	*	&		0.99	&	4404	&	1201	&	*	&	\bf	$>$0.99	\\
9	&	5	&	3086	&	*	&		0.89	&	4813	&	426	&	*	&	\bf	$>$0.99	\\
10	&	24	&	2768	&	*	&		$>$0.99	&	4962	&	34	&	*	&		$>$0.99	\\
11	&	5	&	5450	&	*	&		0.82	&	4867	&	720	&	*	&	\bf	$>$0.99	\\
12	&	8	&	4325	&	*	&	\bf	0.95	&	4947	&	203	&	*	&		0.90	\\
13	&	8	&	3460	&	*	&	\bf	0.94	&	4967	&	87	&	*	&		0.90	\\
14	&	5	&	1709	&	*	&	\bf	0.81	&	4946	&	96	&	*	&		0.80	\\
15	&	8	&	5187	&	*	&		0.98	&	4195	&	2008	&	*	&	\bf	0.98	\\
16	&	18	&	5502	&	*	&		$>$0.99	&	4944	&	59	&	*	&		$>$0.99	\\
17	&	28	&	5168	&	*	&		0.99	&	4870	&	86	&	*	&		0.99	\\\hline
    \end{tabular}
    \caption{Comparison of HOP and LHOP when $p=1$. The `*' symbol in the nF column means the code used all the 10000 allowed function evaluations. For each problem, the largest hypervolumes found by the solvers are highlighted in bold.}
    \label{tab:placeholder_newhvconalfa_nocute}
\end{table}

\begin{table}[htb]
\centering
\small
    \begin{tabular}{l|cccr|cccr}\hline
    & \multicolumn{4}{c|}{HOP} & \multicolumn{4}{c}{LHOP}\\
    P.nr. & it & nP & nF & HV & it & nP & nF & HV \\\hline
    \multicolumn{9}{c}{\texttt{spread\_option}$=1$}\\\hline
1	&	7	&	1589	&	*	&		0.01	&	4944	&	686	&	*	&		0.01	\\
2	&	1191	&	5	&	*	&		0.77	&	4955	&	8	&	*	&	\bf	0.78	\\
3	&	4	&	5363	&	*	&		0.71	&	4869	&	4904	&	*	&		0.71	\\
4	&	3	&	4	&	22	&		0.12	&	9	&	3	&	19	&		0.12	\\
5	&	4	&	998	&	*	&	\bf	0.72	&	4907	&	4929	&	*	&		0.50	\\
6	&	9	&	264	&	*	&	\bf	0.44	&	4951	&	6	&	*	&		0.12	\\
7	&	13	&	1007	&	*	&	\bf	0.73	&	4888	&	4896	&	*	&		0.68	\\
8	&	265	&	24	&	*	&	\bf	0.86	&	7228	&	6	&	*	&		0.84	\\
9	&	17	&	1344	&	*	&	\bf	0.93	&	4981	&	31	&	*	&		0.48	\\
10	&	11	&	1	&	24	&		$>$0.99	&	11	&	1	&	24	&		$>$0.99	\\
11	&	11	&	1320	&	*	&	\bf	0.95	&	4936	&	4983	&	*	&		0.84	\\
12	&	5	&	15	&	61	&	\bf	0.93	&	18	&	10	&	40	&		0.85	\\
13	&	5	&	19	&	71	&		0.82	&	25	&	12	&	56	&	\bf	0.83	\\
14	&	57	&	1037	&	*	&	\bf	0.82	&	7350	&	15	&	*	&		0.80	\\
15	&	9	&	5611	&	*	&	\bf	0.93	&	6	&	1	&	26	&		0.92	\\
16	&	11	&	1	&	33	&	\bf	0.96	&	4957	&	66	&	*	&		0.92	\\
17	&	12	&	445	&	*	&	\bf	0.77	&	4954	&	45	&	*	&		0.01	\\\hline
    \multicolumn{9}{c}{\texttt{spread\_option}$=2$}\\\hline
1	&	7	&	1539	&	*	&		0.01	&	4551	&	125	&	*	&		0.01	\\
2	&	1191	&	5	&	*	&	\bf	0.76	&	4946	&	7	&	*	&		0.74	\\
3	&	3	&	2373	&	*	&		0.71	&	4813	&	141	&	*	&		0.71	\\
4	&	3	&	4	&	22	&		0.12	&	8	&	3	&	18	&		0.12	\\
5	&	4	&	1043	&	*	&	\bf	0.71	&	4794	&	270	&	*	&		0.61	\\
6	&	6	&	1469	&	*	&		0.44	&	4944	&	38	&	*	&	\bf	0.46	\\
7	&	13	&	1061	&	*	&	\bf	0.73	&	4959	&	49	&	*	&		0.17	\\
8	&	80	&	80	&	*	&	\bf	0.87	&	8788	&	6	&	*	&		0.85	\\
9	&	8	&	403	&	*	&	\bf	0.92	&	4981	&	27	&	*	&		0.25	\\
10	&	11	&	1	&	24	&		$>$0.99	&	11	&	1	&	24	&		$>$0.99	\\
11	&	10	&	1538	&	*	&	\bf	0.95	&	4956	&	39	&	*	&		0.86	\\
12	&	5	&	14	&	60	&	\bf	0.93	&	18	&	10	&	40	&		0.85	\\
13	&	3	&	17	&	63	&		0.76	&	27	&	15	&	61	&	\bf	0.78	\\
14	&	196	&	138	&	*	&		0.80	&	4995	&	274	&	*	&		0.80	\\
15	&	10	&	5721	&	*	&	\bf	0.94	&	4892	&	133	&	*	&		0.92	\\
16	&	11	&	1	&	33	&		0.99	&	13	&	25	&	64	&		0.99	\\
17	&	13	&	585	&	*	&	\bf	0.86	&	4954	&	46	&	*	&		0.01	\\\hline
    \end{tabular}
    \caption{Comparison of HOP and LHOP when $p=2$. The `*' symbol in the nF column means the code used all the 10000 allowed function evaluations. For each problem, the largest hypervolumes found by the solvers are highlighted in bold.}
    \label{tab:placeholder_newhvconalfa}
\end{table}
Moreover, the column headers {\em it}, {\em nP}, {\em nF}, {\em HV} represent, respectively, number of iterations, number of nondominated points found, number of function evaluations, and hypervolume of the dominated region.   
In the tables, we highlight in bold the largest hypervolume found by the solvers HOP and LHOP. As it can be observed, with $p=1$, HOP achieves the best hypervolume on 16 out of the 34 test problems, whereas LHOP is the best-performing solver on only 8 problems. When the polynomial order is increased to $p=2$, HOP attains the best hypervolume on 20 out of 34 problems, while LHOP does so on only 4.

It is worth emphasizing that, according to the established complexity results, HOP requires a higher number of function evaluations per iteration than LHOP.
However, HOP consistently provides better hypervolume values.

Moreover, the order $p$ of the approximating polynomials has a clear impact on the ability of both solvers to identify stationary Pareto fronts. In particular, using $p=2$ consistently yields smaller stationarity measures than $p=1$, indicating that high-order approximations are more effective in driving the iterates toward stationarity.

\begin{table}[htb]
    \centering
    \begin{tabular}{l|rr|rr||rr|rr}\hline
    & \multicolumn{4}{c||}{$p=1$} & \multicolumn{4}{c}{$p=2$}\\
    &\multicolumn{2}{c|}{\small{\tt spread\_option}=1}&\multicolumn{2}{c||}{\small{\tt spread\_option}=2}&\multicolumn{2}{c|}{\small{\tt spread\_option}=1}&\multicolumn{2}{c}{\small{\tt spread\_option}=2}\\
    P.nr. &  HOP & LHOP&  HOP & LHOP&  HOP & LHOP&  HOP & LHOP\\
    \hline
 1 &    \bf 0.80 &    0.71  &   \bf  0.80 &    0.77  &    0.01 &    0.01  &    0.01 &    0.01  \\
 2 &    0.80 &   \bf  0.82  &    0.80 &  \bf  0.82  &    0.76 &    0.77  &    0.76 &    0.74  \\
 3 &    \bf 0.73 &    0.72  &   \bf  0.73 &    0.72  &    0.71 &    0.71  &    0.71 &    0.71  \\
 4 &    \bf 0.21 &    $<$0.01  &    $<$0.01 &    $<$0.01  &    $<$0.01 &    $<$0.01  &   $<$0.01 &    $<$0.01  \\
 5 &    \bf 0.66 &    0.30  &    0.64 &    0.47  &    0.55 &    0.38  &    0.55 &    0.47  \\
 6 &    \bf 0.98 &    0.97  &   \bf 0.98 &    0.97  &    0.84 &    0.73  &    0.84 &    0.52  \\
 7 &    0.48 &    0.48  &    0.48 &    0.48  &    0.68 &    0.63  &   \bf 0.71 &    0.24  \\
 8 & \bf $>$0.99 &    0.92  &    0.99 & \bf $>$0.99  & \bf $>$0.99 &    0.92  & \bf $>$0.99 &    0.92  \\
 9 &    0.89 &    0.99  &    0.89 & \bf $>$0.99  &    0.96 &    0.96  &    0.96 &    0.95  \\
10 & \bf $>$0.99 & \bf $>$0.99  & \bf $>$0.99 & \bf $>$0.99  & \bf $>$0.99 & \bf $>$0.99  & \bf $>$0.99 & \bf $>$0.99  \\
11 &    0.82 &    0.99  &    0.82 & \bf $>$0.99  &    0.99 &    0.97  &    0.99 &    0.98  \\
12 & \bf $>$0.99 & \bf $>$0.99  & \bf $>$0.99 & \bf $>$0.99  &    0.99 &    0.98  &    0.99 &    0.98  \\
13 & \bf 0.95 &    0.91  &    0.94 &    0.90  &    0.93 &    0.93  &    0.93 &    0.93  \\
14 & \bf  0.82 &    0.80  &  \bf  0.82 &    0.81  &  \bf  0.82 &    0.81  &  \bf  0.82 &  \bf  0.82  \\
15 &   \bf 0.96 &   \bf 0.96  & \bf   0.96 &  \bf  0.96  &    0.94 &    0.92  &  \bf  0.96 &    0.94  \\
16 & \bf $>$0.99 & \bf $>$0.99  & \bf $>$0.99 & \bf $>$0.99  &    0.41 &    0.39  &    0.41 &    0.41  \\
17 & \bf $>$0.99 & \bf $>$0.99  & \bf $>$0.99 & \bf $>$0.99  &    0.83 &    0.01  &    0.86 &    0.01  \\\hline    \end{tabular}
    \caption{Comparison of the methods with respect to the order $p$ of the models employed. For each problem, the largest hypervolumes found by the solvers are highlighted in bold.}
    \label{tab:placeholder_hv}
\end{table}

To assess the effect of the polynomial order $p$ on the performance of the proposed methods, we aggregate the hypervolume results in Table~\ref{tab:placeholder_hv}. As can be seen, the best hypervolume values are most frequently obtained when $p=1$.

This behavior can be explained by considering the dual task of multiobjective optimization. On the one hand, the goal is to identify Pareto stationary points; on the other hand, it is equally important to generate a well-distributed set of high-quality nondominated solutions that provides a good approximation of the true Pareto front. The latter aspect is effectively captured by the hypervolume indicator.
Consequently, when $p=1$, the methods encourage solution sets with better hypervolume values, even if this comes at the expense of stationarity. In contrast, increasing the polynomial order to $p=2$ shifts the balance toward satisfying the stationarity conditions more accurately. As will be shown in the following, this improvement in stationarity is accompanied by a less pronounced advantage in terms of hypervolume.

To evaluate the Pareto stationarity of the solutions produced by the algorithms, we collect, for each method, all the nondominated points generated at the final iteration. The empirical cumulative distribution function of the corresponding stationarity measures is then reported in Figure~\ref{fig1}.
More specifically, given a solver $s\in\cal S$, 
the cumulative distribution of stationarity measure of the nondominated points obtained by solver $s$ is
\[
d_s(\alpha) = \frac{1}{|{\cal F}_s|}|\{y\in{\cal F}_s:\ y=F(x), \|v(x)\|\leq\alpha\}|,
\]
where  
${\cal F}_s = \bigcup_{p\in{\cal P}} {\cal F}_{p,s}$ and $v(\cdot)$ is defined as in~\eqref{vx}. As shown in Figure~\ref{fig1}, when linear models (i.e., $p=1$) are employed, both variants of HOP are outperformed by LHOP with ${\tt spread\_option}=1$, while they outperform LHOP with ${\tt spread\_option}=2$. This behavior is consistent with the design of the algorithms. In particular, LHOP with ${\tt spread\_option}=2$ always selects the nondominated point with the largest stationarity measure, thereby explicitly prioritizing the reduction of the worst stationarity value. Although this strategy is intended to mimic the behavior of HOP, the available function evaluation budget is often exhausted before comparable stationarity levels can be achieved. Conversely, LHOP with ${\tt spread\_option}=1$ selects points according to the spread criterion, repeatedly improving the stationarity of the same subset of solutions and consequently producing a more favorable overall distribution of stationarity measures.

The advantage of HOP over LHOP becomes more evident when quadratic models (i.e., $p=2$) are employed. As illustrated in Figure~\ref{fig1}, both variants of HOP consistently outperform both variants of LHOP in terms of stationarity. This behavior is expected since, for $p=2$, both methods are designed to reproduce the behavior of Newton's method in the multiobjective optimization setting. The superior performance of HOP can be attributed to its more thorough refinement of the current nondominated set at each iteration, which enables a more effective reduction of the stationarity measure.

Summing up, we could say that HOP generally provides better hypervolume values than LHOP, while LHOP has a lower per-iteration cost.

\begin{figure}
\begin{center}
\includegraphics
[width=0.98\textwidth]{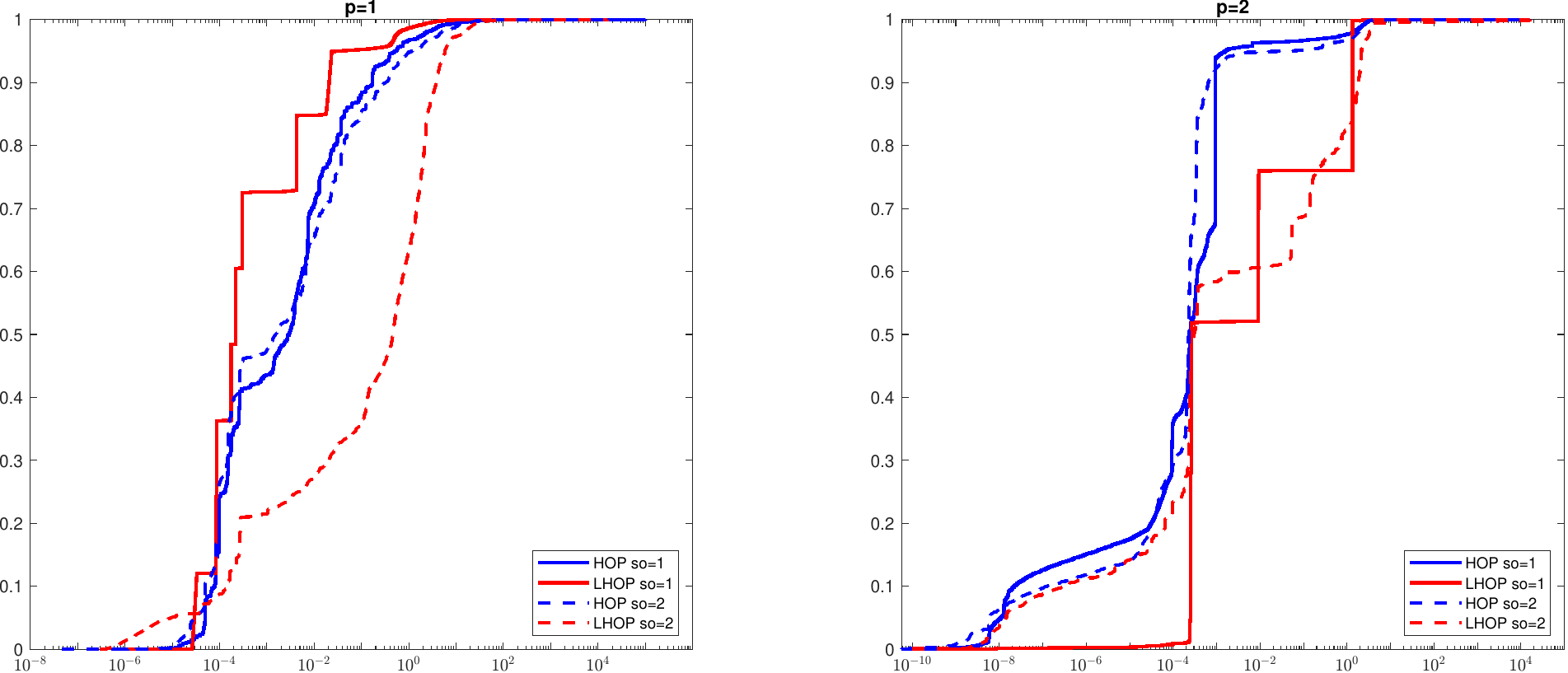}    
    \caption{Distribution of stationarity measure of the nondominated points obtained by the different variants of the algorithms HOP and LHOP. Left plot is for $p=1$, whereas right plot is for $p=2$.}
    \label{fig1}
\end{center}
\end{figure}

\section{Conclusions}\label{sec:conclusions}
In this paper, we have presented an algorithmic framework for Pareto front reconstruction in unconstrained multiobjective optimization, which generates a set of points using high-order regularized models.
At every iteration, the proposed scheme uses a search procedure where several trial points are computed starting from those included in the current set and, if some conditions are satisfied, the trial points can be added to the current set.
We have analyzed both the cases where the regularized models are solved exactly and inexactly, giving worst-case complexity bounds.
Then, a lighter version of the method has also been investigated where, at each iteration, the search procedure is applied only to one point of the current set rather than to all of them.
In the paper, we focused on the general nonconvex case, while the analysis for the convex or strongly convex case may be the subject of future research.

As a final remark, we would like to emphasize that our analysis does not rely on linked sequences, which are commonly used when dealing with a-posteriori algorithms for multiobjective problems (see, e.g., \cite{cocchi:2020,custodio:2021,liuzzi2016derivative,mohammadi2024trust}).
Instead, our analysis uses the fact that, whenever a point is added to the current set, we can relate the Pareto-stationary violation to the increase in the hypervolume of the set in the image space.
This extends a common technique used in the analysis of algorithms for single-objective optimization and algorithms for multiobjective optimization that generate a single Pareto-stationary point, where the stationary violation is typically related to the decrease in the objective function(s).

\clearpage
\bibliography{ref}

\end{document}